\documentclass[12pt,reqno]{amsart}
\usepackage{amsmath}
\usepackage{amssymb}
\usepackage{amsthm}
\usepackage{amstext}
\usepackage{geometry}
\usepackage{fancyhdr}
\usepackage{array}
\usepackage{graphicx}
\usepackage{hyperref}
\hypersetup{
	colorlinks=true,
	linkcolor=blue,
	urlcolor=red,
	citecolor=green,
}

\author{Chunhua Wang, Wenju Wu and Fulin Zhong$^{\dagger}$}
\title{Brouwer degree for Chern-Simons Higgs models on finite graphs}

\address{School of Mathematics and Statistics and Key Laboratory of Nonlinear Analysis and Applications, Ministry of Education, Wuhan 430079, P. R. China}
\email{chunhuawang@ccnu.edu.cn}	

\address{School of Mathematics and Statistics, Central China Normal University, Wuhan 430079, P. R. China}
\email{wjwu@mails.ccnu.edu.cn}

\address{School of Mathematics and Statistics, Central China Normal University, Wuhan 430079, P. R. China}
\email{flzhong@mails.ccnu.edu.cn}	
\thanks{This paper was supported by National Key Research and Development of China (No. 2022YFA1006900) and NSFC (No.~12471106).}
\thanks{$\dagger$Corresponding author: Fulin Zhong.}

\usepackage{fullpage}
\numberwithin{equation}{section}
\newtheorem{theorem}{Theorem}[section]
\newtheorem{remark}[theorem]{Remark}
\newtheorem{lemma}[theorem]{Lemma}

\begin{document}	
	\maketitle
	\pagestyle{fancy}
	\lhead{}
	\chead{}
	\rhead{}
	\lfoot{}
	\cfoot{\thepage}
	\rfoot{}
	\renewcommand{\headrulewidth}{0pt}
	\renewcommand{\footrulewidth}{0pt}

	\begin{abstract}
		Let $G=(V, E)$ be a finite connected graph, where $V$ denotes the set of vertices and $E$ denotes the set of edges. We revisit the following Chern-Simons Higgs
		model,
		\begin{equation*}
			\Delta u=\lambda \mathrm{e}^u\left(\mathrm{e}^u-1\right)+f \ \text {in} \  V,
		\end{equation*}
		where $\Delta$ is the graph Laplacian, $\lambda$ is a real number and $f$ is a function defined on $V$. Firstly, when $\lambda \int_V f \mathrm{d} \mu \neq 0$, we find that the odevity of the number of vertices in the graph affects the number of solutions. Then by calculating the topological degree and using the relationship between the degree and the critical group of a related functional, we obtain the existence of multiple solutions. Also we study the existence of solutions when $\lambda \int_V f \mathrm{d} \mu=0$. These findings extend the work of Huang
		et al. [Comm Math Phys 377:613-621 (2020)], Hou and Sun [Calc Var 61:139 (2022)] and Li et al. [Calc Var 63:81 (2024)]. Similarly, for the generalized Chern-Simons Higgs model, we obtain the same results. Moreover, this method is also applied to the Chern-Simons Higgs system, yielding partial results for the existence of multiple solutions. To our knowledge, this is the first instance where it has been concluded that an equation on graphs can have at least three distinct solutions. We think that our results will be valuable for studying the multiplicity of solutions to analogous equations on graphs.
		
		
		\vspace{0.25cm}
		
	\end{abstract}

	\section{Introduction and main results}
	In this paper, we consider the Chern-Simons Higgs model on a connected finite graph. The model is described as
	\begin{equation} \label{CSH-graph-1.1}
		\Delta u=\lambda \mathrm{e}^u\left(\mathrm{e}^u-1\right)+f \ \text {in} \  V,
	\end{equation}
	where $\Delta$ denotes the graph Laplacian, $\lambda$ is a real number and $f: V \rightarrow \mathbb{R}$ is a function. When $\lambda>0$ and $f=4 \pi \sum_{i=1}^{N} \delta_{P_i}$, \eqref{CSH-graph-1.1} becomes
	\begin{equation} \label{CSH-graph-1.1(1)}
		\Delta u=\lambda \mathrm{e}^u\left(\mathrm{e}^u-1\right)+4\pi \sum_{i=1}^{N} \delta_{P_i},
	\end{equation}
	where $N$ is any fixed positive integer and $P_1,\dots, P_N$ are arbitrarily chosen
	distinct vertices on the graph, and $\delta_{P_i}$ is the Dirac delta function at the vertex $P_i$ for $i=1,\cdots,N$. This model has been the subject of extensive research due to its rich mathematical structure and potential applications in various areas of physics, including condensed matter physics and quantum field theory. Caffarelli and Yang \cite{CY} proved that for \eqref{CSH-graph-1.1(1)} on
	doubly periodic regions in $\mathbb{R}^{2}$ (the 2-tori), there exists a critical value $\lambda_c$ such that if $\lambda>\lambda_c$, \eqref{CSH-graph-1.1(1)} has a solution; if $\lambda<\lambda_c$, \eqref{CSH-graph-1.1(1)} has no solution. Tarantello \cite{T} established the existence of solutions for \eqref{CSH-graph-1.1(1)} if $\lambda=\lambda_c$ and obtained multiple condensate solutions if $\lambda>\lambda_c$. For more information on the Chern-Simons Higgs model one can refer to \cite{CY,CI,DJLW,DJLPW,HKP,JW,LPY,NT1,NT2,T,W} and the references are therein.
	
	Among many research directions, partial differential equations arising in the realms of geometry or physics hold particular significance when studied on graphs. Huang et al. \cite{HLY} obtained an existence result for \eqref{CSH-graph-1.1(1)} on a connected finite graph, aligning with the result of Caffarelli and Yang on the 2-tori. Later, the critical case $\lambda=\lambda_c$ was solved by Hou and Sun \cite{HS}, who showed that \eqref{CSH-graph-1.1(1)} also has a solution. These results are essentially based on the method of super- and sub-solutions principle. It is noteworthy that mathematicians have widely studied various other equations like the heat equation \cite{GJ,HLLY,HX1,LWY,LWY1} and the Schr\"odinger equation \cite{CLZ1,HSM,HSZ,HXW,ZZ}. Furthermore, Grigoryan, Lin and Yang \cite{AYY0,AYY1,AYY2} studied the existence of solutions for nonlinear elliptic equations on graphs by using the variational methods, and topological degree theory (such as in \cite{L,LY,SW} inspired by \cite{LYY1}) is also widely used for this purpose. In particular, Li et al. \cite{LSY} employed topological degree theory to investigate the existence of solutions to \eqref{CSH-graph-1.1} on finite graphs. Moreover, for the generalized Chern-Simons Higgs model, characterized by the equation
	\begin{equation} \label{CSH-graph-1.2}
		\Delta u=\lambda \mathrm{e}^u\left(\mathrm{e}^u-1\right)^{2p-1}+f \ \text {in} \  V,
	\end{equation}
	where $\lambda \in \mathbb{R}$, $f:V \rightarrow \mathbb{R}$ is a function and $p$ is a positive integer. If $p = 1$, \eqref{CSH-graph-1.2} is
	equivalent to \eqref{CSH-graph-1.1}.  When $\lambda>0$, $p=3$ with $f=4 \pi \sum_{i=1}^{N} \delta_{P_i}$, Han \cite{HX} established the existence of multi-vortices for \eqref{CSH-graph-1.2} over a doubly periodic region of $\Omega \subset \mathbb{R}^{2}$. For this special case, Chao and Hou \cite{CH} have proved the existence and multiplicity of solutions to \eqref{CSH-graph-1.2} on a connected finite graph by the super- and sub-solutions principle and the mountain pass theorem. Moreover, Hou and Qiao \cite{HQ} extend the work of Li et al. \cite{LSY} and Chao and Hou \cite{CH} by employing topological degree theory. Our aim is to complement the earlier work of Li et al. \cite{LSY}.

	To describe the Chern-Simons Higgs model in the graph setting, we start by recalling the definition of a graph. Let $V$ denote the set of vertices and $E$ denote the set of edges. We use $G=(V, E)$ to denote a finite graph, where the number of $V$ is finite. We assume that the weights $\{\omega_{xy}: xy \in E\}$ are always positive and symmetric, that is, for each edge $xy \in E$, the weight $\omega_{xy}$ satisfies $\omega_{xy}>0$ and $\omega_{xy}=\omega_{yx}$. For any $xy \in E$, We also assume that they can be connected via a finite number of edges, and then $G$ is called a connected graph. Let $\mu: V \to \mathbb{R}^{+}$ be a finite measure. For any function $u: V \to \mathbb{R}$, the Laplacian is defined by
	$$\Delta u(x)=\frac{1}{\mu(x)} \sum_{y \sim x} \omega_{xy}(u(y)-u(x)),$$
	where $y \sim x$ means $xy \in E$. For a pair of functions $u$ and $v$, the gradient form is defined as
	$$\Gamma(u, v)(x)=\frac{1}{2 \mu(x)} \sum_{y \sim x} \omega_{xy}(u(y)-u(x))(v(y)-v(x)).$$	
	For simplicity, we write $\Gamma(u):=\Gamma(u, u)$. We denote the length of its gradient by
	$$|\nabla u|(x)=\sqrt{\Gamma(u)(x)}=\Big(\frac{1}{2 \mu(x)} \sum_{y \sim x} \omega_{xy}(u(y)-u(x))^{2}\Big)^{1 / 2}.$$	
	For any function $g: V \rightarrow \mathbb{R}$, the integral of $g$ on $V$ is defined by	
	$$\int_V g \mathrm{d} \mu=\sum_{x \in V} \mu(x) g(x),$$
	and an integral average of $g$ is denoted by
	$$\overline{g}=\frac{1}{|V|} \int_V g \mathrm{d} \mu=\frac{1}{|V|} \sum_{x \in V} \mu(x) g(x),$$	
	where $|V|=\sum_{x \in V} \mu(x)$ stands for the volume of $V$. In this paper, we denote $\ell$ as the number of the vertices in the vertex set. The Lebesgue space $L^{\infty}(V)$ is denoted by
	\begin{equation*}
		L^{\infty}(V)=\left\{f:V\rightarrow \mathbb{R} : \|f\|_{L^{\infty}(V)}<+\infty\right\},
	\end{equation*}
	where
	\begin{equation*}
		\|f\|_{L^{\infty}(V)}=\sup_{x\in\Omega}|f(x)|.
	\end{equation*}

	One of our main motivations comes from Li et al.'s seminal work \cite{LSY},
	where they studied the existence and multiplicity of solutions for the Chern-Simons Higgs model \eqref{CSH-graph-1.1}.  Among other things, they proved that
	
	\vspace{0.25cm}
	
	\noindent \textbf{Theorem A}  \cite{LSY} Let $(V, E)$ be a connected finite graph with symmetric weights. Then we have the following:
	\begin{enumerate}
		\item[(a)] If $\lambda \overline{f}<0$, then the Eq.\eqref{CSH-graph-1.1} has a solution.;
		\item[(b)] If $\lambda \overline{f}>0$, then two subcases are distinguished: (i) $\overline{f}>0$. There exists a real number $\Lambda^{*}>0$ such that when $\lambda>\Lambda^{*}$, \eqref{CSH-graph-1.1} has at least two different solutions; when $0<\lambda<\Lambda^{*}$, \eqref{CSH-graph-1.1} has no solution; when $\lambda=\Lambda^{*}$, \eqref{CSH-graph-1.1} has at least one solution; (ii) $\overline{f}<0$. There exists a real number $\Lambda_{*}<0$ such that when $\lambda<\Lambda_{*}$, \eqref{CSH-graph-1.1} has at least two different solutions; when $\Lambda_{*}<\lambda<0$, \eqref{CSH-graph-1.1} has no solution; when $\lambda=\Lambda_{*}$, \eqref{CSH-graph-1.1} has at least one solution.
	\end{enumerate}
	
	Moreover, Hou and Qiao \cite{HQ} proved that the above conclusions hold for \eqref{CSH-graph-1.2}. Based on the above results, we find that the odevity of the number of vertices in the graph affects the number of solutions to \eqref{CSH-graph-1.1} when $\lambda \overline{f} \neq 0$. Consequently, we will supply additional cases concerning the existence of multiple solutions. Furthermore, our another goal here is to extend these results to the case where $\lambda \overline{f}=0$, thereby enriching the known solutions to \eqref{CSH-graph-1.1}. Note that in the case where $\lambda \overline{f}=0$, we cannot use topological degree theory to analyze the solutions to \eqref{CSH-graph-1.1}, since the solution to \eqref{CSH-graph-1.1} may be unbounded, as seen in Remark \ref{CSH-graph-remark-1.1}. Therefore, our first main result is as follows.
	
	\begin{theorem}\label{CSH-graph-theo-1.1}
		Let $(V, E)$ be a connected finite graph with symmetric weights. Then the following statements hold:
		\begin{enumerate}
			\item[(a)] If $\lambda \overline{f}<0$, $\lambda<0$ and $\ell$ is even, then there exists a real number $\Lambda_*^1\leq0$ such that when $\lambda<\Lambda_*^1$, \eqref{CSH-graph-1.1} has at least two distinct solutions;
			\item[(b)] If $\lambda \overline{f}>0$, $\lambda<0$ and $\ell$ is even, then there exist two real numbers $\Lambda_*, \Lambda_*^1<0$ satisfying $\Lambda_*^1\leq\Lambda_*<4\overline{f}$, such that when $\lambda<\Lambda_*^1$, \eqref{CSH-graph-1.1} has at least three distinct solutions;
			\item[(c)] If $\lambda \overline{f}=0$, then four subcases are distinguished:
			\begin{enumerate}
				\item[(i)] if $\lambda = 0$ and $\overline{f}\neq0$, then \eqref{CSH-graph-1.1} has no solution;
				\item[(ii)] if $\lambda=0$ and $\overline{f}=0$, then \eqref{CSH-graph-1.1} has infinite solutions;
				\item[(iii)] if $\lambda > 0$ and $\overline{f}=0$, then there exists a real number $\Lambda^* \geq 0$ such that when $\lambda>\Lambda^*$, \eqref{CSH-graph-1.1} has at least one solution;
				\item[(vi)] if $\lambda < 0$ and $\overline{f}=0$, then there exists a real number $\Lambda_*^1 \leq 0$ such that when $\lambda<\Lambda_*^1$, \eqref{CSH-graph-1.1} has at least one solution.
			\end{enumerate}
		\end{enumerate}
	\end{theorem}
	
	Similarly, regarding the generalized Chern-Simons Higgs model \eqref{CSH-graph-1.2}, we expand upon the work of Hou and Qiao \cite{HQ}. Consequently, we obtain the following result.
	
	\begin{theorem}\label{CSH-graph-theo-1.1(1)}
		Under the same conditions as those stated in Theorem \ref{CSH-graph-theo-1.1}, the solutions to \eqref{CSH-graph-1.2} exhibit the same existence properties.
	\end{theorem}
	
	\begin{remark}
		\begin{enumerate}
			\item[(a)] The definition of $\Lambda_*$ in Theorem \ref{CSH-graph-theo-1.1} is the same as in \cite{LSY}, that is,
			\begin{equation} \label{CSH-graph-2.13}
				\Lambda_*=\sup \left\{\lambda<0: \lambda \overline{f}>0, J_\lambda \text { has a local minimum critical point}\right\}.
			\end{equation}
			\item[(b)] For cases (a) and (b) in Theorem \ref{CSH-graph-theo-1.1}, we will construct either a locally strict minimum or a locally strict maximum solution, and subsequently utilize topological degree theory to establish the existence of an additional solution.
		\end{enumerate}
	\end{remark}
	
	\begin{remark} \label{CSH-graph-remark-1.1}
		If $f\equiv0$, then we consider
		\begin{equation} \label{CSH-graph-z202}
			\Delta u=\lambda \mathrm{e}^u\left(\mathrm{e}^u-\sigma\right) \ \text {in} \  V,
		\end{equation}
		for $\sigma \in [0,1]$ and $\lambda \in \mathbb{R}$. Notice that there exists at least one solution $u_\sigma=\ln \sigma$ for \eqref{CSH-graph-z202}, which satisfies the property
		\begin{equation*}
			\lim\limits_{\sigma \rightarrow 0^+} \left|u_{\sigma}\right|= +\infty.
		\end{equation*}
		Therefore, in the case where $\lambda \overline{f} = 0$, the Brouwer degree defined as in \cite{LSY} may not be well-defined, which precludes the use of its homotopy properties as outlined in \cite{LSY} for analyzing the number of solutions to \eqref{CSH-graph-1.1}.
		
		Here, we construct locally strict maximum and locally strict minimum solutions to \eqref{CSH-graph-1.1} for any given $f \in L^{\infty}(V)$, and utilize the super- and sub-solutions principle to obtain an additional solution, see Lemmas \ref{CSH-graph-lem-2.3}, \ref{CSH-graph-lem-2.4} and \ref{CSH-graph-lem-2.5}. The most important aspect is that our application of the super- and sub-solution principle does not necessitate any additional conditions on $f$, thereby allowing us to derive more results than previously possible.
	\end{remark}
	
	Denote $X=L^{\infty}(V)$ and define a functional $J_\lambda : X \rightarrow \mathbb{R}$ by
	\begin{equation} \label{CSH-graph-z7}
		J_\lambda (u)=\frac{1}{2} \int_V |\nabla u|^2 \mathrm{d} \mu + \frac{\lambda}{2} \int_V (\mathrm{e}^u-1)^{2} \mathrm{d} \mu + \int_V fu \mathrm{d}\mu.
	\end{equation}
	A locally strict minimum (maximum) solution to \eqref{CSH-graph-1.1} is understood to be a local minimal  (maximal) critical point of $J_\lambda $.
	
	Also we consider the generalized Chern-Simons Higgs system
	\begin{equation} \label{CSH-graph-1.3(1)}
		\begin{cases}
			\Delta u=2q\lambda\mathrm{e}^v\left(\mathrm{e}^u-1\right)^{2p}
			\left(\mathrm{e}^v-1\right)^{2q-1}+f \,\,\,&\text{in}\,\,\,V,\\
			\Delta v=2p\lambda\mathrm{e}^u\left(\mathrm{e}^u-1\right)^{2p-1}
\left(\mathrm{e}^v-1\right)^{2q}+g\,\,\,&\text{in}\,\,\,V,
		\end{cases}
	\end{equation}
	where
	$$
	p, q \in \Bigl\{-\frac{1}{2}+t:t \in \mathbb{Z}^{+}=\{1,2,\cdots\}\Bigr\},
	$$
	$\lambda$ is a real number and $f, g \in X$. It was studied by Li et al. \cite{LSY} that if $p=\frac{1}{2}$ and $q=\frac{1}{2}$. From the proofs of our subsequent theorems, it can be seen that once it is determined that $\lambda>0$, it has no impact on the proofs. For simplicity, we only consider the case where $\lambda=1$, that is \eqref{CSH-graph-1.3(1)} becomes
	\begin{equation} \label{CSH-graph-1.3}
		\begin{cases}
			\Delta u=2q\mathrm{e}^v\left(\mathrm{e}^u-1\right)^{2p}
			\left(\mathrm{e}^v-1\right)^{2q-1}+f \,\,&\text{in}\,\,\,V,\\
			\Delta v=2p\mathrm{e}^u\left(\mathrm{e}^u-1\right)^{2p-1}
\left(\mathrm{e}^v-1\right)^{2q}+g\,\,&\text{in}\,\,\,V.
		\end{cases}
	\end{equation}
	The first and most important step is to get a priori estimate for solutions.
	\begin{theorem}\label{CSH-graph-theo-1.3}
		Let $(V, E)$ be a connected finite graph with symmetric weights. Suppose that $\sigma \in[0,1]$, $f$ and $g$ satisfy
		\begin{equation*}
			\Lambda_{1}^{-1} \leq \left|\int_V f \mathrm{d} \mu\right| \leq \Lambda_{1}, \ \Lambda_{1}^{-1} \leq \left|\int_V g \mathrm{d} \mu\right| \leq \Lambda_{1}, \ \|f\|_X \leq \Lambda_{2}, \ \|g\|_X \leq \Lambda_{2}
		\end{equation*}
		for some real number $\Lambda_{1}>0$ and $\Lambda_{2}>0$.
		If $(u, v)$ is a solution to the system
		\begin{equation} \label{CSH-graph-1.4}
			\begin{cases}
				\Delta u=2q\mathrm{e}^v\left(\mathrm{e}^u-\sigma\right)^{2p}
\left(\mathrm{e}^v-\sigma\right)^{2q-1}+f\,\,\,&\text{in}\,\,\,V, \\
				\Delta v=2p\mathrm{e}^u\left(\mathrm{e}^u-\sigma\right)^{2p-1}
\left(\mathrm{e}^v-\sigma\right)^{2q}+g\,\,\,&\text{in}\,\,\,V,
			\end{cases}
		\end{equation}	
		then there exists a constant $C>0$, depending only on $p$, $q$, $\Lambda_1$, $\Lambda_2$ and the graph $V$, such that
		\begin{equation*}
			\|u\|_X+\|v\|_X \leq C.
		\end{equation*}
	\end{theorem}
	
	To compute the topological degree, we define a map $G: X \times X \rightarrow X \times X$ by
	\begin{equation} \label{CSH-graph-1.5}
		G(u, v)=
		\begin{pmatrix}
			-\Delta u+2q\mathrm{e}^v\left(\mathrm{e}^u-1\right)^{2p}\left(\mathrm{e}^v-1\right)^{2q-1}+f\\
			-\Delta v+2p\mathrm{e}^u\left(\mathrm{e}^u-1\right)^{2p-1}\left(\mathrm{e}^v-1\right)^{2q}+g
		\end{pmatrix}^\top ,
	\end{equation}
	where $^\top$ is the transpose of the matrix.
	
	\begin{theorem}\label{CSH-graph-theo-1.4}
		Let $(V, E)$ be a connected finite graph with symmetric weights and $G$ be the map defined by \eqref{CSH-graph-1.5}. If $\overline{f}>0$ and $\overline{g}>0$, then there exists a large number $R_0>0$ such that for all $R \geq R_0$,
		\begin{equation*}
			\deg\left(G, B_R,(0,0)\right)=0,
		\end{equation*}
		where $B_R=\left\{(u, v) \in X \times X:\|u\|_X+\|v\|_X<R\right\}$ is a ball in $X \times X$.
	\end{theorem}
	
	Define a functional $\mathcal{G}: X \times X \rightarrow \mathbb{R}$ by
	\begin{equation}\label{CSH-graph-1.6}
		\mathcal{G}(u, v)=\int_V \nabla u \nabla v \mathrm{d} \mu+ \int_V\left(\mathrm{e}^u-1\right)^{2p}\left(\mathrm{e}^v-1\right)^{2q} \mathrm{d} \mu+\int_V(f v+g u) \mathrm{d} \mu.
	\end{equation}
	Note that for all $(\phi, \psi) \in X \times X$,
	\begin{equation} \label{CSH-graph-12}
		\begin{aligned}
			&\left\langle\mathcal{G}^{\prime}(u, v),(\phi, \psi)\right\rangle =\left.\frac{\mathrm{d}}{\mathrm{d} t}\right|_{t=0} \mathcal{G}(u+t \phi, v+t \psi) \\
			=&\int_V\big[\left(-\Delta v+2p\mathrm{e}^u\left(\mathrm{e}^u-1\right)^{2p-1}\left(\mathrm{e}^v-1\right)^{2q}+g\right) \phi\\
			&\quad +\left(-\Delta u+2q\mathrm{e}^v\left(\mathrm{e}^u-1\right)^{2p}\left(\mathrm{e}^v-1\right)^{2q-1}+f\right) \psi\big] \mathrm{d} \mu.
		\end{aligned}
	\end{equation}
	Clearly, $(u, v)$ is a critical point of $\mathcal{G}$ if and only if it is a solution to the system \eqref{CSH-graph-1.3}. As a consequence of Theorem \ref{CSH-graph-theo-1.4}, we have the following theorem.
	\begin{theorem} \label{CSH-graph-theo-1.5}
		Let $(V, E)$ be a connected finite graph with symmetric weights, $\overline{f}>0$, $\overline{g}>0$ and $\mathcal{G}$ be the functional defined by \eqref{CSH-graph-1.6}. If $\mathcal{G}$ has either a non-degenerate critical point or a locally strict minimum critical point, then it must have another critical point.
	\end{theorem}
	
	\begin{remark}
		Theorem \ref{CSH-graph-theo-1.5} provides an additional solution to system \eqref{CSH-graph-1.3} under the condition that the functional $\mathcal{G}$ has either a non-degenerate critical point or a local minimum critical point beforehand. Consequently, this result is only partial with respect to the problem of finding multiple solutions to system \eqref{CSH-graph-1.3}.
	\end{remark}
	
	Our paper is organized as follows. In Section $2$, we delve into the necessary preliminaries, establishing the strong maximum principle, an elliptic estimate, and other technical tools. Section $3$ is devoted to the rigorous proof of Theorem \ref{CSH-graph-theo-1.1}, which forms a critical component of our analysis. Moving forward to Section $4$, we engage in a thorough discussion of the a priori estimate and explore the existence of solutions for the generalized Chern-Simons Higgs system. This discussion will be supported by the findings presented in Theorems \ref{CSH-graph-theo-1.3}, \ref{CSH-graph-theo-1.4}, and \ref{CSH-graph-theo-1.5}.

	\section{Preliminaries}
	In this section, we introduce several lemmas that will be pivotal in our analysis in Section 3. We start with the following strong maximum principle.
	\begin{lemma} \label{CSH-graph-lem-2.1}
		Suppose that $G=(V,E)$ be a connected finite graph.
		\begin{enumerate}
			\item[(a)] If $u$ is not a constant function, then there exists $x_1 \in V$, such that $u(x_1)=\max_Vu$ and $-\Delta u(x_1)\ge 0$;
			\item[(b)] If $u$ is not a constant function, then there exists $x_2 \in V$, such that $u(x_2)=\min_Vu$ and $-\Delta u(x_2)\le 0$.
		\end{enumerate}
	\end{lemma}
	\begin{proof}
		By the finiteness of the graph and the definition of $\Delta$, we can directly obtain the desired results.
	\end{proof}
	
	\begin{lemma}\label{CSH-graph-lem-a2.3}
		For any $u \in X$, then there exists a constant $C>0$ dependent only on $G$, such that
		\begin{equation*}
			\int_V |\nabla u|^2 \mathrm{d} \mu  \leq C \|u\|_X^2.
		\end{equation*}
	\end{lemma}
	\begin{proof}
		Direct computation gives that
		\begin{equation*}
			\begin{aligned}
				\int_V |\nabla u|^2 \mathrm{d} \mu&=\frac{1}{2}\sum_{x \in V}\sum_{y \sim x}   \omega_{xy}(u(y)-u(x))^{2} \\
				&\leq 2\max_{x,y \in V} \omega_{xy} \|u\|_X^2.
			\end{aligned}
		\end{equation*}
		Thus we can choose $C=2\max_{x,y \in V} \omega_{xy}>0$.
	\end{proof}

	We next prove that for sufficiently large values of $|\lambda|$, \eqref{CSH-graph-1.1} possesses a locally strict minimum solution or a locally strict maximum solution.
	\begin{lemma}\label{CSH-graph-lem-2.3}
		For any given function $f \in X$, if $|\lambda|$ is chosen to be sufficiently large, then the following hold:
		\begin{enumerate}
			\item[(i)] when $\lambda<0$, \eqref{CSH-graph-1.1} has a locally strict maximum solution;
			\item[(ii)] when $\lambda>0$, \eqref{CSH-graph-1.1} has a locally strict minimum solution.
		\end{enumerate}
	\end{lemma}
	\begin{proof}
		
		(i) Let $\varphi$ be the unique solution to the equation
		\begin{equation} \label{CSH-graph-38zz}
			\begin{cases}
				\Delta \varphi=f-\overline{f}, \\
				\int_V \varphi \mathrm{d} \mu=0
			\end{cases}
		\end{equation}
		and define $v:=u-\varphi$.  Then $v$ satisfies
		\begin{equation} \label{CSH-graph-2.1}
			\Delta v=\beta \mathrm{e}^v\left(\mathrm{e}^v-\mathrm{e}^{-\varphi}\right)+\overline{f} \ \text {in} \  V,
		\end{equation}
		where $\beta:=\lambda \mathrm{e}^{2\varphi}$.
		We consider the energy functional
		\begin{equation*}
			Q_\lambda (v)=\frac{1}{2} \int_V |\nabla v|^2 \mathrm{d} \mu + \frac{1}{2} \int_V \beta(\mathrm{e}^v-\mathrm{e}^{-\varphi})^{2} \mathrm{d} \mu+ \int_V \overline{f}v \mathrm{d}\mu, \ v \in X.
		\end{equation*}
		Since $X \cong \mathbb{R}^{\ell}$, $Q_\lambda \in C^2(X, \mathbb{R})$, and for any $A>0$, the set
		\begin{equation*}
			\Big\{v \in X: -\varphi+\ln \frac{1}{2} \leq v \leq -\varphi+A\Big\}
		\end{equation*}
		is a bounded closed subset of $X$, it is straightforward to find some $v_\lambda \in X$ satisfying $-\varphi+\ln \frac{1}{2} \leq v_\lambda(x) \leq -\varphi+A$ for all $x \in V$ and
		\begin{equation} \label{CSH-graph-38z}
			Q_\lambda\left(v_\lambda\right)=\max_{-\varphi+\ln \frac{1}{2} \leq v \leq -\varphi+A} Q_\lambda(v).
		\end{equation}
		Notice that
		\begin{equation*}
			Q_\lambda (-\varphi)=\frac{1}{2} \int_V |\nabla \varphi|^2 \mathrm{d} \mu
		\end{equation*}
		is a constant, which is independent on $\lambda$.
		
		On the one hand, for $\lambda \leq -1$ and $A>2\|\varphi\|_X+1$, if there exists some point $x_\lambda  \in V$ such that $v_\lambda (x_\lambda)=-\varphi(x_\lambda)+A$, then we have
		\begin{equation*}
			v_\lambda (x_\lambda)=\|v_\lambda \|_X>\|\varphi\|_X+1
		\end{equation*}
		and
		\begin{equation*}
			\|\varphi\|_X<-\|\varphi\|_X+A \leq \|v_\lambda \|_X \leq \|\varphi\|_X+A.
		\end{equation*}
		Recall $\lambda < 0$, by applying Lemma \ref{CSH-graph-lem-a2.3} and Cauchy-Schwarz inequality, it holds that
		\begin{equation} \label{CSH-graph-2.3}
			\begin{aligned}
				Q_\lambda (v_\lambda )\leq  &\frac{C}{2} \|v_\lambda \|_X^2 + \frac{\lambda \mathrm{e}^{-2\|\varphi\|_X}}{2} \mu_{\min} (\mathrm{e}^{\|v_\lambda \|_X}-\mathrm{e}^{\|\varphi\|_X})^{2} + |\overline{f}| |V|\|v_\lambda \|_X\\
				\leq &\frac{C}{2} \left(\|\varphi\|_X+A\right)^2 - \frac{ \mathrm{e}^{-2\|\varphi\|_X}}{2} \mu_{\min} (\mathrm{e}^{-\|\varphi\|_X+A}-\mathrm{e}^{\|\varphi\|_X})^{2}+|\overline{f}| |V| \left(\|\varphi\|_X+A\right)\\
				\rightarrow& -\infty, \ \text{as} \ A \rightarrow +\infty, \ \text{uniformly with } \lambda<-1,
			\end{aligned}
		\end{equation}
		where $\mu_{\min}:=\min_V \mu(x)>0$ and $C$ is defined in Lemma \ref{CSH-graph-lem-a2.3}. Hence, there exists a constant $A>0$ such that for any $\lambda<-1$, we have
		\begin{equation*}
			Q_\lambda (v_\lambda ) < Q_\lambda (-\varphi),
		\end{equation*}
		which contradicts with \eqref{CSH-graph-38z}. Consequently, we deduce that
		\begin{equation*}
			v_\lambda(x)<-\varphi+A \ \text{for all} \ x \in V.
		\end{equation*}
		
		On the other hand, if there exists some point $x_\lambda  \in V$ such that $v_\lambda (x_\lambda ) =-\varphi(x_\lambda )+\ln \frac{1}{2}$, then we have
		\begin{equation*}
			\|v_\lambda \|_X \leq \|\varphi\|_X-\ln \frac{1}{2}.
		\end{equation*}
		It can be shown that
		\begin{equation} \label{CSH-graph-2.4}
			\begin{aligned}
				Q_\lambda (v_\lambda )
				&\leq \frac{C}{2} \|v_\lambda \|_X^2  + \frac{\lambda}{2}  \mu(x_\lambda)\big(\mathrm{e}^{\ln \frac{1}{2}}-1\big)^{2} + |\overline{f}| |V|\|v_\lambda \|_X\\
				&\leq \frac{C}{2} \Bigl(\|\varphi\|_X-\ln \frac{1}{2}\Bigr)^2 + \frac{\lambda}{8} \mu_{\min}+ |\overline{f}| |V| \Bigl(\|\varphi\|_X-\ln \frac{1}{2}\Bigr)\\
				& \rightarrow -\infty, \ \text{as} \ \lambda \rightarrow -\infty.
			\end{aligned}
		\end{equation}
		Thus there exists a constant $\lambda_0<0$ such that for any $\lambda<\lambda_0$, we have
		\begin{equation*}
			Q_\lambda (v_\lambda ) < Q_\lambda (-\varphi).
		\end{equation*}
		This is a contradiction to \eqref{CSH-graph-38z}. Therefore, for any $\lambda<\min\{-1,\lambda_0\}$, it follows that
		\begin{equation} \label{CSH-graph-2.4z}
			Q_\lambda\left(v_\lambda\right)=\max_{-\varphi+\ln \frac{1}{2} \leq v \leq -\varphi+A} Q_\lambda(v)=\max_{-\varphi+\ln \frac{1}{2} < v < -\varphi+A} Q_\lambda(v).
		\end{equation}
		We conclude that $v_\lambda$ is a locally strict maximum critical point of $Q_\lambda$. In particular, $v_\lambda+\varphi$ is a solution to \eqref{CSH-graph-1.1}. This implies that $v_\lambda+\varphi$ is a locally strict maximum solution to \eqref{CSH-graph-1.1}. In fact, since $\varphi$ satisfies \eqref{CSH-graph-38zz}, for any $v \in X$, we have
		\begin{equation*}
			\int_V \nabla \varphi \nabla v \mathrm{d}\mu =-\int_V  v\left(f-\overline{f}\right) \mathrm{d}\mu
		\end{equation*}
		and
		\begin{equation*}
			\int_V |\nabla \varphi|^2\mathrm{d}\mu =-\int_V  \varphi\left(f-\overline{f}\right) \mathrm{d}\mu=-\int_V  \varphi f  \mathrm{d}\mu.
		\end{equation*}
		Then the following relation holds:
		\begin{equation}\label{bueq.1}
			\begin{aligned}
				J_\lambda (v+\varphi)=& \frac{1}{2} \int_V |\nabla (v+\varphi)|^2 \mathrm{d} \mu + \frac{\lambda}{2} \int_V (\mathrm{e}^{v+\varphi}-1)^{2} \mathrm{d} \mu + \int_V f(v+\varphi) \mathrm{d}\mu \\
				=& Q_\lambda (v)+\int_V \nabla \varphi \nabla v \mathrm{d}\mu +\int_V  v\left(f-\overline{f}\right) \mathrm{d}\mu+\frac{1}{2}\int_V |\nabla \varphi|^2\mathrm{d}\mu +\int_V  \varphi f  \mathrm{d}\mu \\
				=& Q_\lambda (v)-\frac{1}{2}\int_V |\nabla \varphi|^2\mathrm{d}\mu.
			\end{aligned}
		\end{equation}
		Therefore, by \eqref{CSH-graph-2.4z} and \eqref{bueq.1}, it holds that
		\begin{equation*}
			\begin{aligned}
				J_\lambda (v_\lambda +\varphi)&=Q_\lambda\left(v_\lambda\right)-\frac{1}{2}\int_V |\nabla \varphi|^2\mathrm{d}\mu\\
				&=\max_{-\varphi+\ln \frac{1}{2} < v < -\varphi+A} \Bigl(Q_\lambda(v)-\frac{1}{2}\int_V |\nabla \varphi|^2\mathrm{d}\mu\Bigr)=\max_{\ln \frac{1}{2} < u < A} J_\lambda(u).
			\end{aligned}
		\end{equation*}

		(ii) Define the operator $L_\lambda:X \rightarrow X$ by
		\begin{equation*} \label{CSH-graph-36}
			L_\lambda (u)=-\Delta u+\lambda \mathrm{e}^u\left(\mathrm{e}^u-1\right)+f.
		\end{equation*}
		For any real numbers $A$ and $\lambda$, there hold
		\begin{equation*}
			L_\lambda (A)=\lambda \mathrm{e}^A\left(\mathrm{e}^A-1\right)+f, \ L_\lambda \Bigl(\ln \frac{1}{2}\Bigr)=-\frac{1}{4} \lambda+f.
		\end{equation*}
		Clearly, by choosing sufficiently large $A>1$ and $\lambda>1$, we obtain
		\begin{equation} \label{CSH-graph-37}
			L_\lambda (A)>0, \ L_\lambda \Bigl(\ln \frac{1}{2}\Bigr)<0.
		\end{equation}
		Recall the functional $J_\lambda: X \rightarrow \mathbb{R}$ defined by \eqref{CSH-graph-z7}. Since $X \cong \mathbb{R}^{\ell}, J_\lambda \in C^2(X, \mathbb{R})$, and the set $\left\{u \in X: \ln \frac{1}{2} \leq u \leq A\right\}$ is a bounded closed subset of $X$, it is straightforward to find some $u_\lambda \in X$ satisfying $\ln \frac{1}{2} \leq u_\lambda(x) \leq A$ for all $x \in V$ and
		\begin{equation} \label{CSH-graph-38}
			J_\lambda\left(u_\lambda\right)=\min _{\ln \frac{1}{2} \leq u \leq A} J_\lambda(u).
		\end{equation}
		We claim that
		\begin{equation} \label{CSH-graph-39}
			\ln \frac{1}{2}<u_\lambda(x)<A,\ \text {for all} \ x \in V.
		\end{equation}
		Suppose not, then there must hold $u_\lambda\left(x_0\right)=\ln \frac{1}{2}$ for some $x_0 \in V$, or $u_\lambda\left(x_1\right)=A$ for some $x_1 \in V$. If $u_\lambda\left(x_0\right)=\ln \frac{1}{2}$, we take a small $\varepsilon>0$ such that
		\begin{equation*}
			\ln \frac{1}{2} \leq u_\lambda(x)+t \delta_{x_0}(x) \leq A, \ \text {for all} \ x \in V, t \in(0, \varepsilon),
		\end{equation*}
		where for any $y \in V$, the function $\delta_y(x)$ is defined as
		\begin{equation*}
			\delta_{y}(x):=\begin{cases}
				\frac{1}{\mu(y)} & \ \text{if} \  x=y, \\
				0 & \ \text{if} \  x \neq y.
			\end{cases}
		\end{equation*}
		On the one hand, in view of \eqref{CSH-graph-37} and \eqref{CSH-graph-38}, we have
		\begin{equation} \label{CSH-graph-40}
			\begin{aligned}
				0 & \leq\left.\frac{\mathrm{d}}{\mathrm{d} t}\right|_{t=0} J_\lambda\left(u_\lambda+t \delta_{x_0}\right) \\
				& =\int_V\left(-\Delta u_\lambda+\lambda \mathrm{e}^{u_\lambda}\left(\mathrm{e}^{u_\lambda}-1\right)+f\right) \delta_{x_0} \mathrm{d} \mu \\
				& =-\Delta u_\lambda\left(x_0\right)+\lambda \mathrm{e}^{u_\lambda\left(x_0\right)}\left(\mathrm{e}^{u_\lambda\left(x_0\right)}-1\right)+f\left(x_0\right) \\
				& <-\Delta u_\lambda\left(x_0\right).
			\end{aligned}
		\end{equation}
		On the other hand, since $u_\lambda(x) \geq u_\lambda\left(x_0\right)$ for all $x \in V$ and using Lemma \ref{CSH-graph-lem-2.1}, we conclude that $\Delta u_\lambda\left(x_0\right) \geq 0$, which contradicts \eqref{CSH-graph-40}. Therefore, $u_\lambda(x)>\ln \frac{1}{2}$ for all $x \in V$.
		
		In the same way, we exclude the possibility of $u_\lambda\left(x_1\right)=A$ for some $x_1 \in V$. If $u_\lambda\left(x_1\right)=A$, we take a small $\varepsilon>0$ such that
		\begin{equation*}
			\ln \frac{1}{2} \leq u_\lambda(x)-t \delta_{x_1}(x) \leq A, \ \text {for all} \ x \in V, t \in(0, \varepsilon).
		\end{equation*}
		On the one hand, in view of \eqref{CSH-graph-37} and \eqref{CSH-graph-38}, we have
		\begin{equation} \label{CSH-graph-140}
			\begin{aligned}
				0 & \leq\left.\frac{\mathrm{d}}{\mathrm{d} t}\right|_{t=0} J_\lambda\left(u_\lambda-t \delta_{x_1}\right) \\
				& =\int_V\left(-\Delta u_\lambda+\lambda \mathrm{e}^{u_\lambda}\left(\mathrm{e}^{u_\lambda}-1\right)+f\right) \left(-\delta_{x_1}\right) \mathrm{d} \mu \\
				& =\Delta u_\lambda\left(x_1\right)-\lambda \mathrm{e}^{u_\lambda\left(x_1\right)}\left(\mathrm{e}^{u_\lambda\left(x_1\right)}-1\right)-f\left(x_1\right) \\
				& <\Delta u_\lambda\left(x_1\right).
			\end{aligned}
		\end{equation}
		On the other hand, since $u_\lambda(x) \leq A$ for all $x \in V$ and by applying Lemma \ref{CSH-graph-lem-2.1}, we deduce that $\Delta u_\lambda\left(x_1\right) \leq 0$, which contradicts \eqref{CSH-graph-140}.
		This confirms our claim \eqref{CSH-graph-39}. Combining \eqref{CSH-graph-38} with \eqref{CSH-graph-39}, we conclude that $u_\lambda$ is a locally strict minimum critical point of $J_\lambda$. In particular, $u_\lambda$ is a solution to \eqref{CSH-graph-1.1}. This implies $u_\lambda$ is a locally strict minimum solution to \eqref{CSH-graph-1.1}.
	\end{proof}
	
	\begin{remark}
		The method used in (i) is different from that in \cite[Lemma 10]{LSY}. In fact, the proof of (ii) can also follow the method used in (i), which we omit. Here we prove (ii) by the method in \cite[Lemma 10]{LSY} for any given $f \in X$.
	\end{remark}
	
	To proceed, we also need the following lemma.
	\begin{lemma} \label{CSH-graph-lem-2.4}
		If equation $L_{0} u=0$ has a solution $u_{0}$ for any given function $f \in X$, then for any $\lambda>0$ and $\kappa_1 \in \left(0,\lambda\right)$ small enough, we have
		\begin{equation*}
			L_\lambda\left(u_{0}+\ln \frac{\kappa_1}{\lambda}\right)<0.
		\end{equation*}
	\end{lemma}
	\begin{proof}
		If $\lambda>\kappa_1>0$, then
		\begin{equation*}
			\begin{aligned}
				L_\lambda\left(u_{0}+\ln \frac{\kappa_1}{\lambda}\right) & =-\Delta u_{0}+\kappa_1 \mathrm{e}^{u_{0}}\left(\frac{\kappa_1}{\lambda} \mathrm{e}^{u_{0}}-1\right)+f \\
				& <-\Delta u_{0}+f \\
				& =0,
			\end{aligned}
		\end{equation*}
		since $\frac{\kappa_1}{\lambda} \mathrm{e}^{u_{0}}-1 \leq 0$, if we take $\kappa_1$ small enough.
	\end{proof}
	
	As a consequence, we have the following lemma.
	\begin{lemma} \label{CSH-graph-lem-2.5}
		Assume that $L_{0} u_{0}=0$ on $V$ for any given function $f \in X$. If $\lambda>0$, then \eqref{CSH-graph-1.1} has a locally strict minimum solution $u_\lambda$.
	\end{lemma}
	\begin{proof}
		Notice that $\kappa_1$ is defined in Lemma \ref{CSH-graph-lem-2.4}. Let $A>1$ be a sufficiently large constant such that $L_\lambda A>0$ and $u_{0}+\ln \frac{\kappa_1}{\lambda}<A$ on $V$. Then there exists some $u_\lambda \in X$ such that
		\begin{equation*}
			J_\lambda\left(u_\lambda\right)=\min _{u_{0}+\ln \frac{\kappa_1}{\lambda} \leq u \leq A} J_\lambda(u).
		\end{equation*}
		Suppose that there is some point $x_0 \in V$ satisfying $u_\lambda\left(x_0\right)=u_{0}\left(x_0\right)+\ln \frac{\kappa_1}{\lambda}$. Let $\varepsilon>0$ be so small that for $t \in(0, \varepsilon)$, there holds
		\begin{equation*}
			u_{0}(x)+\ln \frac{\kappa_1}{\lambda} \leq u_\lambda(x)+t \delta_{x_0}(x) \leq A,\ \text {for all} \ x \in V.
		\end{equation*}
		Similar to what we do in the proof of Lemma \ref{CSH-graph-lem-2.3} and according to Lemma \ref{CSH-graph-lem-2.4}, we have
		\begin{equation*}
			\begin{aligned}
				0 & \leq\left.\frac{\mathrm{d}}{\mathrm{d} t}\right|_{t=0} J_\lambda\left(u_\lambda+t \delta_{x_0}\right) \\
				& =-\Delta u_\lambda\left(x_0\right)+\lambda \mathrm{e}^{u_\lambda\left(x_0\right)}\left(\mathrm{e}^{u_\lambda\left(x_0\right)}-1\right)+f\left(x_0\right) \\
				& =-\Delta\left(u_\lambda-u_{0}\right)\left(x_0\right)+L_\lambda\Bigl(u_{0}+\ln \frac{\kappa_1}{\lambda}\Bigr)\left(x_0\right) \\
				& <-\Delta\left(u_\lambda-u_{0}\right)\left(x_0\right).
			\end{aligned}
		\end{equation*}
		This contradicts the fact that $x_0$ is a minimum point of $u_\lambda-u_{0}-\ln \frac{\kappa_1}{\lambda}$. Hence
		\begin{equation*}
			u_\lambda(x)>u_{0}(x)+\ln \frac{\kappa_1}{\lambda}, \ \text {for all} \ x \in V.
		\end{equation*}
		In the same way, we find that $u_\lambda(x)<A$ for all $x \in V$. In fact,
		suppose that there exists a point $x_1 \in V$ such that $u_\lambda\left(x_1\right)=A$. Letting $\varepsilon>0$ be sufficiently small so that for $t \in(0, \varepsilon)$, it follows that
		\begin{equation} \label{CSH-graph-z111}
			u_{0}(x)+\ln \frac{\kappa_1}{\lambda} \leq u_\lambda(x)-t \delta_{x_1}(x) \leq A,\ \text {for all} \ x \in V.
		\end{equation}
		Similar to what we do in the proof of Lemma \ref{CSH-graph-lem-2.3}, we have
		\begin{equation*}
			\begin{aligned}
				0 & \leq\left.\frac{\mathrm{d}}{\mathrm{d} t}\right|_{t=0} J_\lambda\left(u_\lambda-t \delta_{x_1}\right) \\
				& =\Delta u_\lambda\left(x_1\right)-\lambda \mathrm{e}^{u_\lambda\left(x_1\right)}\left(\mathrm{e}^{u_\lambda\left(x_1\right)}-1\right)-f\left(x_1\right) \\
				& =\Delta u_\lambda\left(x_1\right)-L_\lambda\left(A\right) \\
				& <\Delta u_\lambda\left(x_1\right).
			\end{aligned}
		\end{equation*}
		This contradicts the fact that $x_1$ is a maximum point of $u_\lambda$. Hence
		\begin{equation} \label{CSH-graph-z112}
			u_\lambda(x)<A, \ \text {for all} \ x \in V.
		\end{equation}
		Therefore, $u_\lambda$ is a locally strict minimum critical point of $J_\lambda$. Thus we complete the proof of the lemma.
	\end{proof}

	We conclude from Lemmas \ref{CSH-graph-lem-2.3} and \ref{CSH-graph-lem-2.5} that the following two critical numbers are well defined. For any given $f \in X$, we define
	\begin{equation} \label{CSH-graph-2.14}
		\Lambda_*^1=\Lambda_*^1(f):=\sup \left\{\lambda<0: J_\lambda \text { has a locally strict maximum critical point}\right\}
	\end{equation}
	and
	\begin{equation} \label{CSH-graph-2.15}
		\Lambda^*=\Lambda^*(f):=\inf \left\{\lambda>0: J_\lambda \text { has a locally strict minimum critical point}\right\}.
	\end{equation}
	
	\begin{lemma} \label{CSH-graph-lem-2.8}
		For any given function $f \in X$, we have the following:
		\begin{enumerate}
			\item[(i)] if $\overline{f}>0$, then $\Lambda^*\geq 4\overline{f}$; if $\overline{f}<0$, then $\Lambda_* \leq 4\overline{f}$;
			\item[(ii)] if $\overline{f}\geq 0$, then $\Lambda_*^1 \leq0$; if $\overline{f}\leq 0$, then $\Lambda^*\geq 0$;
			\item[(iii)] if $\overline{f}<0$, then $\Lambda_*^1 \leq \Lambda_*$.
		\end{enumerate}
	\end{lemma}
	\begin{proof}
		(i) The proof is derived from \cite[Lemma 13]{LSY};
		
		(ii) By \eqref{CSH-graph-2.14} and \eqref{CSH-graph-2.15}, the conclusion immediately holds;
		
		(iii) If not, suppose that $\Lambda_*^1 > \Lambda_*$ and there exists a number $\lambda_1 \in\left(\Lambda_*,\Lambda_*^1\right)$ such that \eqref{CSH-graph-1.1} has a solution at $\lambda=\lambda_1$. According to \cite[Lemma 12]{LSY}, \eqref{CSH-graph-1.1} would then have a local minimum solution for any $\lambda \in\left[\Lambda_*, \lambda_1\right)$. This contradicts the definition of $\Lambda_*$. Therefore, it must hold $\Lambda_*^1 \leq \Lambda_*$.
	\end{proof}
	
	We have the following elliptic estimate.
	\begin{lemma}\label{CSH-graph-lem-2.9}
		Let $(V, E)$ be a connected finite graph with symmetric weights. Then there exists a constant $C>0$ depending only on the graph $V$, such that for any function $u \in X$, it holds that
		\begin{equation*} \label{CSH-graph-18}
			\max _V {u}-\min _V u \leq C\|\Delta u\|_X.
		\end{equation*}
	\end{lemma}
	\begin{proof}
		Assume that $V=\left\{x_1, \cdots, x_{\ell}\right\}$ such that $u\left(x_1\right)=\max _V u$ and $u\left(x_{\ell}\right)=\min _V u$. Without loss of generality, let $x_1 x_2, x_2 x_3, \cdots, x_{\ell-1} x_{\ell}$ be the shortest path connecting $x_1$ and $x_{\ell}$. It follows that
		\begin{equation} \label{CSH-graph-2.16}
			\begin{aligned}
				0 \leq u\left(x_1\right)-u\left(x_{\ell}\right) & \leq \sum_{j=1}^{\ell-1}\left|u\left(x_j\right)-u\left(x_{j+1}\right)\right| \\
				& \leq \frac{\sqrt{\ell-1}}{\sqrt{w_0}}\Bigl(\sum_{j=1}^{\ell-1} w_{x_j x_{j+1}}\left(u\left(x_j\right)-u\left(x_{j+1}\right)\right)^2\Bigr)^{\frac{1}{2}} \\
				& \leq \frac{\sqrt{\ell-1}}{\sqrt{w_0}}\Bigl(\int_V|\nabla u|^2 \mathrm{d} \mu\Bigr)^{\frac{1}{2}}.
			\end{aligned}
		\end{equation}
		Denoting $\lambda_{1}:=\inf _{\overline{v}=0, \int_V v^{2} \mathrm{d} \mu=1} \int_V|\nabla v|^{2} \mathrm{d} \mu$ and $w_{0}:= \min _{x \in V, y \sim x} \omega_{xy}$, we obtain by integration by parts
		\begin{equation*}
			\begin{aligned}
				\int_V|\nabla u|^2 \mathrm{d} \mu & =-\int_V(u-\overline{u}) \Delta u \mathrm{d} \mu \\
				& \leq\Bigl(\int_V(u-\overline{u})^2 \mathrm{d} \mu\Bigr)^{\frac{1}{2}}\Bigl(\int_V(\Delta u)^2 \mathrm{d} \mu\Bigr)^{\frac{1}{2}} \\
				& \leq\Bigl(\frac{1}{\lambda_1} \int_V|\nabla (u-\overline{u})|^2 \mathrm{d} \mu\Bigr)^{\frac{1}{2}}\Bigl(\int_V(\Delta u)^2 \mathrm{d} \mu\Bigr)^{\frac{1}{2}}\\
				& =\Bigl(\frac{1}{\lambda_1} \int_V|\nabla u|^2 \mathrm{d} \mu\Bigr)^{\frac{1}{2}}\Bigl(\int_V(\Delta u)^2 \mathrm{d} \mu\Bigr)^{\frac{1}{2}},
			\end{aligned}
		\end{equation*}
		which gives
		\begin{equation}\label{CSH-graph-17}
			\int_V|\nabla u|^2 \mathrm{d} \mu \leq \frac{1}{\lambda_1} \int_V(\Delta u)^2 \mathrm{d} \mu \leq \frac{1}{\lambda_1}\|\Delta u\|_X^2|V|.
		\end{equation}		
		Combining \eqref{CSH-graph-2.16} with \eqref{CSH-graph-17}, we conclude that
		\begin{equation*}
			\max _V {u}-\min _V u \leq \sqrt{\frac{(\ell-1)|V|}{w_0 \lambda_1}}\|\Delta u\|_X.
		\end{equation*}
		Choosing $C:=\sqrt{\frac{(\ell-1)|V|}{w_0 \lambda_1}}>0$, thus we complete the proof of the Lemma \ref{CSH-graph-lem-2.9}.
	\end{proof}
	
	Let $J_\lambda : X \rightarrow \mathbb{R}$ be a functional defined as in \eqref{CSH-graph-z7}. Note that the critical point of $J_\lambda $ is a solution to the Chern-Simons equation \eqref{CSH-graph-1.1}. The following property of $J_\lambda $ will be not only useful for our subsequent analysis, but also of its independent interest.
	\begin{lemma}\label{CSH-graph-lem-2.10}
		Under the assumption that $\lambda \overline{f} \neq 0$, the functional $J_\lambda$ satisfies the Palais-Smale condition at any level $c \in \mathbb{R}$.
	\end{lemma}
	\begin{proof}
		Let $c \in \mathbb{R}$ and $\left\{u_k\right\}$ be a sequence in $X$ such that
		\begin{equation*}
			\begin{cases}
				J_\lambda\left(u_k\right) \rightarrow c, \\
				J_\lambda^{\prime}\left(u_k\right) \rightarrow 0 \ \text {in} \  X^* \cong \mathbb{R}^{\ell},
			\end{cases} \ \text{ as $k \rightarrow \infty$.}
		\end{equation*}
		It follows from the definition that
		\begin{equation*}
			-\Delta u_k+\lambda\mathrm{e}^{u_k}\left(\mathrm{e}^{u_k}-1\right)+f=o_k(1),
		\end{equation*}
		where $o_k(1) \rightarrow 0$ uniformly on $V$ as $k \rightarrow \infty$. We define $f_k:=f-o_k(1)$, and for sufficiently large $k$, we also have $\overline{f_k}\neq0$. Moreover, we deduce that
		\begin{equation} \label{CSH-graph-200}
			-\Delta u_k+\lambda\mathrm{e}^{u_k}\left(\mathrm{e}^{u_k}-1\right)+f_k=0.
		\end{equation}
		
		If $u_k$ is a solution to \eqref{CSH-graph-200}, then by integrate by parts, we have
		\begin{equation}\label{CSH-graph-13}
			0=\int_V \Delta u_k \mathrm{d} \mu=\lambda \int_V \mathrm{e}^{u_k}\left(\mathrm{e}^{u_k}-1\right) \mathrm{d} \mu+\int_V f_k \mathrm{d} \mu.
		\end{equation}
		
		Firstly, we show that $\left\{u_k\right\}$ has a uniform upper bound. If $\max _V {u_k}\leq0$, then it is trivial since ${u_k}$ already has an upper bound of $0$.
		If $\max _V {u_k}>0$, observing that
		\begin{equation*}
			\left|\int_{{u_k}<0} \mathrm{e}^{u_k}\left(\mathrm{e}^{u_k}-1\right) \mathrm{d} \mu\right|=\int_{{u_k}<0} \mathrm{e}^{u_k}\left(1-\mathrm{e}^{u_k}\right) \mathrm{d} \mu \leq |V|,
		\end{equation*}
		then we derive from \eqref{CSH-graph-13} that
		\begin{equation*}
			\int_{u_k \geq 0} \mathrm{e}^{u_k}\left(\mathrm{e}^{u_k}-1\right) \mathrm{d} \mu \leq |V|+\frac{1}{|\lambda|}\left|\int_V f_k \mathrm{d} \mu\right| := a_1.
		\end{equation*}
		This together with the fact that
		\begin{equation*}
			\begin{aligned}
				\int_{u_k \geq 0} \mathrm{e}^{u_k}\left(\mathrm{e}^{u_k}-1\right) \mathrm{d} \mu&=\sum_{x \in V, {u_k}(x) \geq 0} \mu(x) \mathrm{e}^{{u_k}(x)}\left(\mathrm{e}^{{u_k}(x)}-1\right)\\
				& \geq \mu_0 \mathrm{e}^{\max _V {u_k}}\left(\mathrm{e}^{\max _V {u_k}}-1\right)
			\end{aligned}
		\end{equation*}
		leads to
		\begin{equation} \label{CSH-graph-14}
			\max _V u_k \leq \ln \frac{1+\sqrt{1+4 a_1 / \mu_0}}{2},
		\end{equation}
		where $\mu_0=\min _{x \in V} \mu(x)>0$, since $V$ is finite.
		In both cases, we have
		\begin{equation*}
			\max _V u_k \leq \max\bigg\{0,\ln \frac{1+\sqrt{1+4 a_1 / \mu_0}}{2}\bigg\}=\ln \frac{1+\sqrt{1+4 a_1 / \mu_0}}{2}.
		\end{equation*}
		
		Secondly, we prove that $\left\{u_k\right\}$ has also a uniform lower bound. To see this, in view of \eqref{CSH-graph-200} and \eqref{CSH-graph-14}, we calculate for any $x \in V$,
		\begin{equation*}
			\begin{aligned}
				|\Delta {u_k}(x)| & \leq|\lambda|\left|\mathrm{e}^{{u_k}(x)}\left(\mathrm{e}^{{u_k}(x)}-1\right)\right|+|f_k(x)| \\
				& \leq|\lambda|\left(\mathrm{e}^{2 {u_k}(x)}+\mathrm{e}^{{u_k}(x)}\right)+|f_k(x)| \\
				& \leq |\lambda|\bigg(\frac{\big(1+\sqrt{1+4 a_1 / \mu_0}\big)^2}{4}+\frac{1+\sqrt{1+4 a_1 / \mu_0}}{2}\bigg)+\|f_k\|_X := b_{1}.
			\end{aligned}
		\end{equation*}
		Hence, there holds
		\begin{equation} \label{CSH-graph-15}
			\|\Delta {u_k}\|_X \leq b_{1}.
		\end{equation}
		In view of \eqref{CSH-graph-15} and Lemma \ref{CSH-graph-lem-2.9}, we have
		\begin{equation} \label{CSH-graph-19}
			\max _V {u_k}-\min _V u_k \leq c_0:=b_1 \sqrt{\frac{(\ell-1)|V|}{w_0 \lambda_1}}.
		\end{equation}
		Coming back to \eqref{CSH-graph-13}, we observe that
		\begin{equation} \label{CSH-graph-20}
			\int_V \mathrm{e}^{u_k}\left(\mathrm{e}^{u_k}-1\right) \mathrm{d} \mu=-\frac{1}{\lambda} \int_V f_k \mathrm{d} \mu := c_1\neq 0.
		\end{equation}
		Now we claim that
		\begin{equation} \label{CSH-graph-21}
			\max _V {u_k}>-A_1:=\ln \min \left\{1, \frac{\left|c_1\right|}{4|V|}\right\}.
		\end{equation}
		Otherwise, $\max _V u_k \leq-A_1$, which together with \eqref{CSH-graph-20} implies
		\begin{equation*}
			\begin{aligned}
				\left|c_1\right| & =\left|\int_V \mathrm{e}^{u_k}\left(\mathrm{e}^{u_k}-1\right) \mathrm{d} \mu\right| \leq \int_V\left(\mathrm{e}^{2 {u_k}}+\mathrm{e}^{u_k}\right) \mathrm{d} \mu \\
				& \leq\left(\mathrm{e}^{2 \max _V {u_k}}+\mathrm{e}^{\max _V {u_k}}\right)|V| \\
				& \leq 2 \mathrm{e}^{-A_1}|V| \leq \frac{\left|c_1\right|}{2}.
			\end{aligned}
		\end{equation*}
		This contradicts $c_1 \neq 0$. Thus our claim \eqref{CSH-graph-21} is true. Inserting \eqref{CSH-graph-21} into \eqref{CSH-graph-19} and combining it with \eqref{CSH-graph-14}, we obtain
		\begin{equation*}
			-A_1-c_0 \leq \min _V u_k \leq \max _V u_k \leq \ln \frac{1+\sqrt{1+4 a_1 / \mu_0}}{2}.
		\end{equation*}
		Thus we obtain $\left\{u_k\right\}$ is uniformly bounded. Since $X$ is pre-compact, up to a subsequence, $\left\{u_k\right\}$ converges uniformly to some $u^* \in X$, which is a critical point $J_\lambda$. Thus the Palais-Smale condition follows.
	\end{proof}

	\section{Proofs of theorems \ref{CSH-graph-theo-1.1} and  \ref{CSH-graph-theo-1.1(1)}}
	In this section, we aim to prove Theorems \ref{CSH-graph-theo-1.1} and \ref{CSH-graph-theo-1.1(1)} by applying the above lemmas. To prove Theorem \ref{CSH-graph-theo-1.1}, first we state the result of topological degree in \cite{LSY}, which will be used later.
	\begin{lemma}\label{lm3.1}\cite[Theorem 2]{LSY}
		Let $(V, E)$ be a connected finite graph with symmetric weights, and $F: X \rightarrow X$ be a map defined by $F(u)=-\Delta u+\lambda \mathrm{e}^{u}\left(\mathrm{e}^{u}-1\right)+f$. Suppose that $\lambda \overline{f} \neq 0$. Then there exists a large number $R_{0}>0$ such that for all $R \geq R_{0}$,
		\begin{equation*}
			\operatorname{deg}\left(F, B_{R}, 0\right)=\begin{cases}
				1, & \text { if } \lambda>0, \overline{f}<0, \\
				0, & \text { if } \lambda \overline{f}>0, \\
				-1, & \text { if } \lambda<0, \overline{f}>0,
			\end{cases}
		\end{equation*}
		where $B_{R}=\left\{u \in X:\|u\|_{X}<R\right\}$ is a ball in $X$.
	\end{lemma}
We are in a position to prove Theorem \ref{CSH-graph-theo-1.1}.
	\begin{proof}[\textbf{Proof of Theorem \ref{CSH-graph-theo-1.1}.}]
		(a) If $\lambda < \Lambda_*^1$, by \eqref{CSH-graph-2.14}, then we let $u_\lambda$ be a locally strict maximum critical point of $J_\lambda$. Without loss of generality, we may assume that $u_\lambda$ is the unique critical point of $J_\lambda$. Otherwise, $J_\lambda$ would already have at least two critical points, and the proof would be complete. According to \cite[Chapter 1, page 32]{CKC}, the $r$-th critical group of $J_\lambda$ at $u_\lambda$ is defined by
		\begin{equation} \label{CSH-graph-44}
			C_r\left(J_\lambda, u_\lambda\right)=H_r\left(J_\lambda^c \cap U,\left\{J_\lambda^c \backslash\left\{u_\lambda\right\}\right\} \cap U, \mathbb{Z}\right),
		\end{equation}
		where $J_\lambda\left(u_\lambda\right)=c, J_\lambda^c=\left\{u \in X: J_\lambda(u) \leq c\right\}, U$ is a neighborhood of $u_\lambda \in X$ and $H_r$ is the singular homology group with the coefficients group $\mathbb{Z}$. By the excision property of $H_r$, this definition is not dependent on the choice of $U$. Since $u_\lambda$ is a locally strict maximum critical point of $J_\lambda$, it is easy to calculate
		\begin{equation} \label{CSH-graph-45}
			C_r\left(J_\lambda, u_\lambda\right)=\delta_{r \ell} \mathbb{Z}=\begin{cases}
				\mathbb{Z}, & \ \text{if} \  r=\ell, \\
				\{0\}, & \ \text{if} \  r \neq \ell.
			\end{cases}
		\end{equation}
		Note that $J_\lambda$ satisfies the Palais-Smale condition as established in Lemma \ref{CSH-graph-lem-2.10} and
		\begin{equation*}
			D J_\lambda(u)=-\Delta u+\lambda \mathrm{e}^u\left(\mathrm{e}^u-1\right)+f=F(u),
		\end{equation*}
		where $F$  is given as in Lemma \ref{lm3.1}. According to \cite[Chapter 1, Theorem 3.2]{CKC}, in view of \eqref{CSH-graph-45}, for sufficiently large $R>1$, we have
		\begin{equation*}
			\deg\left(F, B_R, 0\right)=\deg\left(D J_\lambda, B_R, 0\right)=\sum_{r=0}^{\infty}(-1)^r \operatorname{rank} C_r\left(J_\lambda, u_\lambda\right)=1,
		\end{equation*}
		since $\ell$ is even. This contradicts $\deg\left(F, B_R, 0\right)=-1$ obtained in Lemma \ref{lm3.1}. Therefore, \eqref{CSH-graph-1.1} must has at least two distinct solutions.
		
		(b) If $\lambda < \Lambda_*^1$ and $\overline{f}<0$, by \eqref{CSH-graph-2.13}, \eqref{CSH-graph-2.14} and Lemma \ref{CSH-graph-lem-2.8}, we let $u_{1\lambda}$ and $u_{2\lambda}$ be a locally strict maximum and locally strict minimum critical point of $J_\lambda$ respectively. Without loss of generality, we may assume that $u_{1\lambda}$ and $u_{2\lambda}$ are the only two critical points of $J_\lambda$. If not, $J_\lambda$ would already at least three critical points, and the proof would be finished. According to \cite[Chapter 1, page 32]{CKC}, the $r$-th critical group of $J_\lambda$ at $u_{i\lambda}$ for $i=1,2$ is defined by
		\begin{equation} \label{CSH-graph-44zz}
			C_r\left(J_\lambda, u_{i\lambda}\right)=H_r\left(J_\lambda^{c_i} \cap U_i,\left\{J_\lambda^{c_i} \backslash\left\{u_{i\lambda}\right\}\right\} \cap U_i, \mathbb{Z}\right),
		\end{equation}
		where $J_\lambda\left(u_{i\lambda}\right)=c_i, J_\lambda^{c_i}=\left\{u \in X: J_\lambda(u) \leq c_i\right\}$, $U_i$ is a neighborhood of $u_{i\lambda} \in X$ and $H_r$ is the singular homology group with the coefficients group $\mathbb{Z}$. It is easy to calculate
		\begin{equation} \label{CSH-graph-45zz}
			C_r\left(J_\lambda, u_{1\lambda}\right)=\delta_{r \ell} \mathbb{Z}=\begin{cases}
				\mathbb{Z}, & \ \text{if} \  r=\ell, \\
				\{0\}, & \ \text{if} \  r \neq \ell
			\end{cases}
		\end{equation}
		and
		\begin{equation} \label{CSH-graph-45z}
			C_r\left(J_\lambda, u_{2\lambda}\right)=\delta_{r 0} \mathbb{Z}=\begin{cases}
				\mathbb{Z}, & \ \text{if} \  r=0, \\
				\{0\}, & \ \text{if} \  r \neq 0.
			\end{cases}
		\end{equation}
		Similar to the proof of (a) and by combining \eqref{CSH-graph-45zz} with \eqref{CSH-graph-45z}, for sufficiently large $R>1$, it holds that
		\begin{equation*}
			\begin{aligned}
				\deg\left(F, B_R, 0\right)=&\deg\left(D J_\lambda, B_R, 0\right)\\
				=&\sum_{r=0}^{\infty}(-1)^r \operatorname{rank} C_r\left(J_\lambda, u_{1\lambda}\right)+\sum_{r=0}^{\infty}(-1)^r \operatorname{rank} C_r\left(J_\lambda, u_{2\lambda}\right)\\
				=&2,
			\end{aligned}
		\end{equation*}
		since $\ell$ is even. This contradicts $\deg\left(F, B_R, 0\right)=0$ obtained in Lemma \ref{lm3.1}. Therefore, \eqref{CSH-graph-1.1} must has at least three distinct solutions for $\lambda < \Lambda_*^1$.
		
		(c) (i) Since $\lambda=0$ and \eqref{CSH-graph-1.1} can be simplified to $\Delta u=f \ \text {in} \  V$, integration by parts yields
		$$0=\int_V \Delta u \mathrm{d}\mu=\int_V f \mathrm{d}\mu.$$
		Hence, \eqref{CSH-graph-1.1} has no solution for $\overline{f}\neq 0$.
		
		(ii) Let $\varphi$ be the unique solution to the equation
		\begin{equation*}
			\begin{cases}
				\Delta \varphi=f-\overline{f}  \ \text {in} \  V, \\
				\int_V \varphi \mathrm{d} \mu=0,
			\end{cases}
		\end{equation*}
		and define $v:=u-\varphi$.  Direct computation gives that
		$\int_V |\nabla v|^2\mathrm{d}\mu =0$.
		Therefore, $v=C$ and $u=\varphi+C$ for any constant $C$.
		
		(iii) According the definition \eqref{CSH-graph-2.15}, we can conclude that \eqref{CSH-graph-1.1} has at least one solution for $\lambda > \Lambda^*$ and $\overline{f}=0$.
		
		(iv) By \eqref{CSH-graph-2.14}, \eqref{CSH-graph-1.1} has at least one solution for $\lambda<\Lambda_*^1$.
	\end{proof}
	
	Now we will prove Theorem \ref{CSH-graph-theo-1.1(1)}.
	\begin{proof}[\textbf{Proof of Theorem \ref{CSH-graph-theo-1.1(1)}.}]
		Since we can prove Theorem \ref{CSH-graph-theo-1.1(1)} only by making some minor modifications of the proof to Theorem \ref{CSH-graph-theo-1.1}, we omit it.
	\end{proof}

	\section{Generalized Chern-Simons Higgs System}
	In this section, we aim to calculate the topological degree of the map associated with the generalized Chern-Simons Higgs system given by \eqref{CSH-graph-1.3}. Subsequently, we will utilize this degree to derive partial results concerning the multiplicity of solutions to the system. Specifically, Theorems \ref{CSH-graph-theo-1.3}, \ref{CSH-graph-theo-1.4}, and \ref{CSH-graph-theo-1.5} will be demonstrated. To begin, we present an elliptic estimate.
	
	\begin{lemma}\label{CSH-graph-lem-4.1}
		Let $\varphi$ be the unique solution to	the equation
		\begin{equation*}
			\begin{cases}
				\Delta \varphi=f \ \text {in} \  V,\\
				\int_V \varphi \mathrm{d} \mu=0.
			\end{cases}
		\end{equation*}
		Then there exists a constant $\tilde{C}>0$ depending only on the graph $V$, such that
		\begin{equation*}
			\|\varphi\|_X \leq \tilde{C} \|f\|_X.
		\end{equation*}
	\end{lemma}
	\begin{proof}
		For any $x \in V$, we calculate
		\begin{equation*}
			|\Delta \varphi(x)|=|f(x)|\leq \|f\|_X.
		\end{equation*}
		Hence, it follows that $\|\Delta \varphi\|_X\leq \|f\|_X$. By Lemma \ref{CSH-graph-lem-2.9}, there exists a constant $\tilde{C}>0$ depending only on the graph $V$, such that
		$$\|\varphi\|_X \leq \tilde{C} \|f\|_X.$$
	\end{proof}
	
	We now derive a priori estimate of solutions to \eqref{CSH-graph-1.4}, which is a deformation of \eqref{CSH-graph-1.3}.
	
	\begin{proof}[\textbf{Proof of Theorem \ref{CSH-graph-theo-1.3}.}]
		Note that there exists a unique solution $\varphi$ to the equation
		\begin{equation*}
			\begin{cases}
				\Delta \varphi=f-\overline{f}  \ \text {in} \  V,\\
				\int_V \varphi \mathrm{d} \mu=0
			\end{cases}
		\end{equation*}
		and a unique solution $\psi$ to the equation
		\begin{equation*}
			\begin{cases}
				\Delta \psi=g-\overline{g} \ \text {in} \  V, \\
				\int_V \psi \mathrm{d} \mu=0.
			\end{cases}
		\end{equation*}
		By Lemma \ref{CSH-graph-lem-4.1}, there exists a constant $\tilde{C}>0$ depending only on the graph $V$, such that
		\begin{equation} \label{CSH-graph-z102}
			\|\varphi\|_X ,\|\psi\|_X \leq \tilde{C} \left(\Lambda_2+\Lambda_3\right),
		\end{equation}
		where $\Lambda_3=\frac{\Lambda_1}{|V|}$. Set $w=u-\varphi$ and $z=v-\psi$. Then we have
		\begin{equation} \label{CSH-graph-46}
			\begin{cases}
				\Delta w=2q \mathrm{e}^\psi \mathrm{e}^z\left(\mathrm{e}^{\varphi} \mathrm{e}^w-\sigma\right)^{2p}\left(\mathrm{e}^\psi \mathrm{e}^z-\sigma\right)^{2q-1}+\overline{f}  \,\,\, &\text {in} \  V, \\
				\Delta z=2p \mathrm{e}^{\varphi} \mathrm{e}^w\left(\mathrm{e}^{\varphi} \mathrm{e}^w-\sigma\right)^{2p-1}\left(\mathrm{e}^\psi \mathrm{e}^z-\sigma\right)^{2q}+\overline{g} \,\,\,& \text {in} \  V.
			\end{cases}
		\end{equation}
		Let $\sigma \in[0,1]$ and $(w, z)$ be a solution to the system \eqref{CSH-graph-46}.
		We claim that
		\begin{equation} \label{CSH-graph-47}
			z(x)\leq-\min _V \psi,\ \text {for all} \ x \in V.
		\end{equation}
		Suppose not, then there necessarily holds $\max _V z >-\min _V \psi$. Take $x_0 \in V$ satisfying $z\left(x_0\right)=\max _V z$. Since $\sigma \in[0,1], \overline{g}>0$ and $\psi\left(x_0\right)+z\left(x_0\right) > 0$, we have
		\begin{equation*}
			\begin{aligned}
				0 &\geq \Delta z\left(x_0\right)=2p \mathrm{e}^{\varphi\left(x_0\right)} \mathrm{e}^{w\left(x_0\right)}\left(\mathrm{e}^{\varphi\left(x_0\right)} \mathrm{e}^{w\left(x_0\right)}-\sigma\right)^{2p-1}\left(\mathrm{e}^{\psi\left(x_0\right)} \mathrm{e}^{z\left(x_0\right)}-\sigma\right)^{2q}+\overline{g}\\
				& \geq 2p \mathrm{e}^{\varphi\left(x_0\right)} \mathrm{e}^{w\left(x_0\right)}\left(\mathrm{e}^{\varphi\left(x_0\right)} \mathrm{e}^{w\left(x_0\right)}-1\right)^{2p-1}\left(\mathrm{e}^{\psi\left(x_0\right)} \mathrm{e}^{z\left(x_0\right)}-1\right)^{2q}+\overline{g}\\
				& \geq \overline{g}>0,
			\end{aligned}
		\end{equation*}
		which is a contradiction. Hence, our claim \eqref{CSH-graph-47} follows. Keeping in mind $\overline{f}>0$, in the same way as above, we also have
		\begin{equation} \label{CSH-graph-48}
			w(x)\leq-\min _V \varphi,\ \text {for all} \ x \in V.
		\end{equation}
		In fact, if not, then there necessarily hold $\max _V w >-\min _V \varphi$. Taking $x_1 \in V$ satisfying $w\left(x_1\right)=\max _V w$, we have
		\begin{equation*}
			\begin{aligned}
				0 &\geq \Delta w\left(x_1\right)=2q \mathrm{e}^{\psi\left(x_1\right)} \mathrm{e}^{z\left(x_1\right)}\left(\mathrm{e}^{\varphi\left(x_1\right)} \mathrm{e}^{w\left(x_1\right)}-\sigma\right)^{2p}\left(\mathrm{e}^{\psi\left(x_1\right)} \mathrm{e}^{z\left(x_1\right)}-\sigma\right)^{2q-1}+\overline{f}\\
				& \geq 2q \mathrm{e}^{\psi\left(x_1\right)} \mathrm{e}^{z\left(x_1\right)}\left(\mathrm{e}^{\varphi\left(x_1\right)} \mathrm{e}^{w\left(x_1\right)}-1\right)^{2p}\left(\mathrm{e}^{\psi\left(x_1\right)} \mathrm{e}^{z\left(x_1\right)}-1\right)^{2q-1}+\overline{f}\\
				& \geq \overline{f}>0,
			\end{aligned}
		\end{equation*}
		which is a contradiction.
		
		By substituting \eqref{CSH-graph-47} and \eqref{CSH-graph-48} into \eqref{CSH-graph-46}, we obtain
		\begin{equation*}
			\|\Delta w\|_X, \|\Delta z\|_X \leq 2\max\{p,q\} \mathrm{e}^{2\tilde{C} (\Lambda_2+\Lambda_3)} \big(\mathrm{e}^{2\tilde{C} \left(\Lambda_2+\Lambda_3\right)} +1\big)^{2p+2q-1}+\Lambda_3.
		\end{equation*}
		Thus using \eqref{CSH-graph-z102}, we have
		\begin{equation*}
			\|\Delta u\|_X, \|\Delta v\|_X \leq b:= 2\max\{p,q\} \mathrm{e}^{2\tilde{C} \left(\Lambda_2+\Lambda_3\right)} \big(\mathrm{e}^{2\tilde{C} \left(\Lambda_2+\Lambda_3\right)} +1\big)^{2p+2q-1}+\Lambda_3+\tilde{C} \left(\Lambda_2+\Lambda_3\right),
		\end{equation*}
		By Lemma \ref{CSH-graph-lem-2.9}, we immediately conclude that
		\begin{equation} \label{CSH-graph-49}
			\max _V {u}-\min _V u , \max _V {v}-\min _V v \leq c:=\sqrt{\frac{(\ell-1)|V|}{w_0 \lambda_1}}b.
		\end{equation}
		Observe that integrating both sides of the first equation in \eqref{CSH-graph-46} yields
		\begin{equation*}
			\int_V \mathrm{e}^\psi \mathrm{e}^z\left(\mathrm{e}^\varphi \mathrm{e}^w-\sigma\right)^{2p}\left(\mathrm{e}^\psi \mathrm{e}^z-\sigma\right)^{2q-1} \mathrm{d} \mu=-\frac{\overline{f}}{2q}|V|.
		\end{equation*}
		As a consequence, we have
		\begin{equation*}
			\begin{aligned}
				\frac{\Lambda_1^{-1}}{2q|V|}<\frac{\overline{f}}{2q} \leq \big(\mathrm{e}^{2\tilde{C} \left(\Lambda_2+\Lambda_3\right)}+1\big)^{2p}\mathrm{e}^{\max _V v}\left(\mathrm{e}^{\max _V v}+1\right)^{2q-1}.
			\end{aligned}
		\end{equation*}
		Hence, there exists a constant $C_1>0$, depending only on $\Lambda_2, \Lambda_3$ and the graph $V$, such that
		\begin{equation*}
			\max_V v \geq -C_1.
		\end{equation*}
		In view of \eqref{CSH-graph-49}, it follows that
		\begin{equation} \label{CSH-graph-z50}
			\min _V v \geq -C_1-c.
		\end{equation}
		In the same way, integrating both sides of the first equation in \eqref{CSH-graph-46} leads to
		\begin{equation*}
			\int_V \mathrm{e}^\varphi \mathrm{e}^w\left(\mathrm{e}^\varphi \mathrm{e}^w-\sigma\right)^{2p-1}\left(\mathrm{e}^\psi \mathrm{e}^z-\sigma\right)^{2q} \mathrm{d} \mu=-\frac{\overline{g}}{2p}|V|.
		\end{equation*}
		As a consequence, there holds
		\begin{equation*}
			\begin{aligned}
				\frac{\Lambda_1^{-1}}{2p|V|}<\frac{\overline{g}}{2p} \leq \big(\mathrm{e}^{2\tilde{C} \left(\Lambda_2+\Lambda_3\right)}+1\big)^{2q}\mathrm{e}^{\max _V u}\left(\mathrm{e}^{\max _V u}+1\right)^{2p-1}.
			\end{aligned}
		\end{equation*}
		We derive that there exists a constant $C_2>0$, depending only on $\Lambda_2, \Lambda_3$ and the graph $V$, such that
		\begin{equation*}
			\max_V u \geq -C_2.
		\end{equation*}
		By using \eqref{CSH-graph-49}, it holds that
		\begin{equation} \label{CSH-graph-z51}
			\min _V u \geq -C_2-c.
		\end{equation}
		In view of \eqref{CSH-graph-z102}, \eqref{CSH-graph-z50} and \eqref{CSH-graph-z51}, we have
		\begin{equation*}
			\begin{aligned}
				\|u\|_X+\|v\|_X
				\leq& \max\big\{4\tilde{C} \left(\Lambda_2+\Lambda_3\right), C_1+c, C_2+c\big\}.
			\end{aligned}
		\end{equation*}
	\end{proof}

	Now, we proceed to calculate the topological degree of the map as defined in \eqref{CSH-graph-1.5}.
	
	\begin{proof}[\textbf{Proof of Theorem \ref{CSH-graph-theo-1.4}.}]
		Define a map $G: X \times X \times[0,1] \rightarrow X \times X$ by
		\begin{equation*}
			G(u, v, \sigma) =
			\begin{pmatrix}
				-\Delta u+2q\mathrm{e}^v\left(\mathrm{e}^u-\sigma\right)^{2p}\left(\mathrm{e}^v-\sigma\right)^{2q-1}+f\\
				-\Delta v+2p\mathrm{e}^u\left(\mathrm{e}^u-\sigma\right)^{2p-1}\left(\mathrm{e}^v-\sigma\right)^{2q}+g
			\end{pmatrix}^\top
		\end{equation*}
		for any $(u, v, \sigma) \in X \times X \times[0,1]$. Obviously, it holds that $G \in C^2(X \times X \times[0,1], X \times X)$. On the one hand, by Theorem \ref{CSH-graph-theo-1.3}, there exists some $R_0>0$ such that for any $R \geq R_0$, we have
		\begin{equation*}
			0 \notin G\left(\partial B_R, \sigma\right), \ \text {for all} \ \sigma \in[0,1].
		\end{equation*}
		Consequently, the homotopic invariance of the topological degree implies
		\begin{equation} \label{CSH-graph-52}
			\deg\left(G(\cdot, 1), B_R,(0,0)\right)=\deg\left(G(\cdot, 0), B_R,(0,0)\right).
		\end{equation}
		Here we denote the open ball in $X\times X$ with radius $R$ as
		\begin{equation*}
			B_R=\left\{(u, v) \in X \times X:\|u\|_X+\|v\|_X<R\right\}
		\end{equation*}
		and its boundary as
		\begin{equation*}
			\partial B_R=\left\{(u, v) \in X \times X:\|u\|_X+\|v\|_X=R\right\}.
		\end{equation*}
		On the other hand, we calculate $\deg\left(G(\cdot, 0), B_R,(0,0)\right)$. Since $\overline{f},\overline{g}>0$, integrating both sides of the first equation of the system
		\begin{equation} \label{CSH-graph-4.10}
			\begin{cases}
				\Delta u=2q\mathrm{e}^{2pu} \mathrm{e}^{2qv}+f  \,\,\,& \text {in} \  V,\\
				\Delta v=2p\mathrm{e}^{2pu} \mathrm{e}^{2qv}+g \,\,\, &\text {in} \  V,
			\end{cases}
		\end{equation}
		we get a contradiction, provided that \eqref{CSH-graph-4.10} is solvable. This implies
		\begin{equation*}
			\{(u, v) \in X \times X: G(u, v, 0)=(0,0)\}=\emptyset.
		\end{equation*}
		As a consequence, we have
		\begin{equation} \label{CSH-graph-54}
			\deg\left(G(\cdot, 0), B_R,(0,0)\right)=0.
		\end{equation}
		Combining \eqref{CSH-graph-52} with \eqref{CSH-graph-54}, we get the desired result.
	\end{proof}
	
	Let $\mathcal{G}: X \times X \rightarrow \mathbb{R}$ be a functional defined as in \eqref{CSH-graph-1.6}. Note that the critical point of $\mathcal{G}$ is a solution to the Chern-Simons system \eqref{CSH-graph-1.3}. Similar to the proof of Lemma \ref{CSH-graph-lem-2.10}, we have the following lemma.
	
	\begin{lemma} \label{CSH-graph-lemma-5.1}
		Under the assumptions $\lambda>0, \overline{f}>0$ and $\overline{g}>0$, it holds that $\mathcal{G}$ satisfies the Palais-Smale condition at any level $c \in \mathbb{R}$.
	\end{lemma}
	\begin{proof}
		Let $c \in \mathbb{R}$ and $\left\{\left(u_k, v_k\right)\right\}$ be a sequence in $X \times X$ such that $\mathcal{G}\left(u_k, v_k\right) \rightarrow c$ and
		\begin{equation*}
			\mathcal{G}^{\prime}\left(u_k, v_k\right) \rightarrow(0,0) \ \text {in} \ (X \times X)^* \cong \mathbb{R}^{\ell} \times \mathbb{R}^{\ell}.
		\end{equation*}
		This together with \eqref{CSH-graph-12} gives
		\begin{equation} \label{CSH-graph-55}
			\begin{cases}
				-\Delta u_k+2q\mathrm{e}^{v_k}\left(\mathrm{e}^{u_k}-1\right)^{2p}\left(\mathrm{e}^{v_k}-1\right)^{2q-1}+f=o_k(1)  \,\,\,&\text {in} \  V,\\
				-\Delta v_k+2p\mathrm{e}^{u_k}\left(\mathrm{e}^{u_k}-1\right)^{2p-1}\left(\mathrm{e}^{v_k}-1\right)^{2q}+g=o_k(1) \,\,\,&\text {in} \  V,
			\end{cases}
		\end{equation}
		where $o_k(1) \rightarrow 0$ uniformly on $V$ as $k \rightarrow \infty$. Comparing \eqref{CSH-graph-55} with the system \eqref{CSH-graph-1.3} and using the similar method as in the proof of Theorem \ref{CSH-graph-theo-1.3}, we have
		\begin{equation*}
			\left\|u_k\right\|_X+\left\|v_k\right\|_X \leq C
		\end{equation*}
		for some constant $C>0$, provided that $k \geq k_1$ for some sufficiently large $k_1 \in \mathbb{Z}^{+}$. Since $V$ is finite, $X$ is pre-compact. Therefore, up to a subsequence, $u_k \rightarrow u^*$ and $v_k \rightarrow v^*$ uniformly in $V$ for some functions $u^* \in X$ and $v^* \in X$. Clearly, $\mathcal{G}^{\prime}\left(u^*, v^*\right)=(0,0)$. Thus $\mathcal{G}$ satisfies the (PS)$_c$ condition.
	\end{proof}

	Finally, we prove a partial multiple solutions result for the system \eqref{CSH-graph-1.3}.
	
	\begin{proof}[\textbf{Proof of Theorem \ref{CSH-graph-theo-1.5}.}]
		We distinguish two hypotheses to proceed.
		
		\textbf{Case 1.} Assume that $\mathcal{G}$ has a non-degenerate critical point $\left(u, v\right)$.
		Since $\left(u, v\right)$ is non-degenerate, we have
		\begin{equation*}
			\det D^2 \mathcal{G}\left(u, v\right) \neq 0.
		\end{equation*}
		Suppose that $\left(u, v\right)$ is the unique critical point of $\mathcal{G}$. Then by the definition of the Brouwer degree, we conclude that for all $R>\left\|u\right\|_X+\left\|v\right\|_X$,
		\begin{equation} \label{CSH-graph-56}
			\deg\left(D \mathcal{G}, B_R,(0,0)\right)=\operatorname{sgn} \det D^2 \mathcal{G}\left(u, v\right) \neq 0.
		\end{equation}
		Here and in the sequel, as in the proof of Theorem \ref{CSH-graph-theo-1.1}, $B_R$ is a ball centered at $(0,0)$ with radius $R$. Notice that $D \mathcal{G}(u, v)=G(u, v)$ for all $(u, v) \in X \times X$, where $G$ is defined as in \eqref{CSH-graph-1.5}. By Theorem \ref{CSH-graph-theo-1.4}, it holds that
		\begin{equation*}
			\deg\left(D \mathcal{G}, B_R,(0,0)\right)=\deg\left(G, B_R,(0,0)\right)=0,
		\end{equation*}
		contradicting \eqref{CSH-graph-56}. Hence, $\mathcal{G}$ must have at least two critical points.
		
		\textbf{Case 2.} Suppose that $\mathcal{G}$ has a locally strict minimum critical point $\left(\varphi, \psi\right)$.
		Without loss of generality, we may assume $u$ is the unique critical point of $\mathcal{G}$. 	Otherwise, $\mathcal{G}$ already has at least two critical points, and the proof is complete. According to \cite[Chapter 1, page 32]{CKC}, the $r$-th critical group of $\mathcal{G}$ at the critical point $\left(\varphi, \psi\right)$ is given by
		\begin{equation*}
			C_r\left(\mathcal{G},\left(\varphi, \psi\right)\right)=H_r\left(\mathcal{G}^c \cap U,\left\{\mathcal{G}^c \backslash\left\{\left(\varphi, \psi\right)\right\}\right\} \cap U, \mathbb{Z}\right),
		\end{equation*}
		where $\mathcal{G}^c=\left\{(u, v) \in X \times X: \mathcal{G}(u, v) \leq c\right\}$, $U$ is a neighborhood of $\left(\varphi, \psi\right) \in X \times X$ and $\mathbb{Z}$ is the coefficient group of $H_r$. Since $\left(\varphi, \psi\right)$ is a locally strict minimum critical point, we easily find that
		\begin{equation*}
			C_r\left(\mathcal{G},\left(\varphi, \psi\right)\right)=\delta_{r 0} \mathbb{Z}=\begin{cases}
				\mathbb{Z}, & \ \text{if} \  r=0, \\
				\{0\}, & \ \text{if} \  r\neq 0.
			\end{cases}
		\end{equation*}
		By Lemma \ref{CSH-graph-lemma-5.1}, $\mathcal{G}$ satisfies the Palais-Smale condition. Then applying \cite[Chapter 1, Theorem 3.2]{CKC} and Theorem \ref{CSH-graph-theo-1.4}, we obtain
		\begin{equation*}
			\begin{aligned}
				0=\deg\left(G, B_R,(0,0)\right) & =\deg\left(D \mathcal{G}, B_R,(0,0)\right) \\
				& =\sum_{r=0}^{\infty}(-1)^r \operatorname{rank} C_r\left(\mathcal{G},\left(\varphi, \psi\right)\right) \\
				& =1,
			\end{aligned}
		\end{equation*}
		provided that $R>\left\|\varphi\right\|_X+\left\|\psi\right\|_X$, which is impossible. Thus $\mathcal{G}$ must have another critical point.
	\end{proof}


\begin{thebibliography}{99}
		
		\bibitem{CY}
		{Caffarelli, L.A, Yang, Y.S.:}{Vortex condensation in the Chern-Simons Higgs model: An existence theorem.}
		\sl Comm. Math. Phys.
		\rm 168, 321-336 (1995).
		
		\bibitem{CI}
		{ Chae, D., Imanuvilov, O.Y.:}{ The existence of non-topological multivortex solutions in the relativistic
			self-dual Chern-Simons theory.} \sl Comm. Math. Phys. \rm 215, 119-142 (2000).
		
		\bibitem{CKC}
		{Chang, K.C.:}
		{Infinite dimensional Morse theory and multiple solution problems.}
		\sl Birkh\"auser, Boston
		\rm (1993).
		
		\bibitem{CH}
		{Chao, R., Hou, S.:}
		{Multiple solutions for a generalized Chern-Simons equation on graphs.}
		\sl J. Math. Anal. Appl.
		\rm 519, 126787 (2023).
		
		
		\bibitem{CLZ1}
		{Chow, S.N., Li, W.C., Zhou, H.M.:}
		{A discrete Schr\"odinger equation via optimal transport on graphs.}
		\sl J. Funct. Anal.
		\rm 276, 2440-2469 (2019).
		
		
		\bibitem{DJLW}	
		{Ding, W., Jost, J., Li, J., Wang, G.:} {An analysis of the two-vortex case in the Chern-Simons Higgs model.}
		\sl Calc. Var. Partial Differential Equations
		\rm 7(1), 87-97 (1998).
		
		\bibitem{DJLPW}		
		{ Ding, W., Jost, J., Li, J., Peng, X., Wang, G.:}{ Self duality equations for Ginzburg-Landau and Seiberg-Witten type functionals with 6th order potentials.}
		\sl Comm. Math. Phys.
		\rm 217, 383-407 (2001).
		
		
		\bibitem{GJ}
		{Ge, H., Jiang, W.:}
		{Kazdan-Warner equation on infinite graphs.}
		\sl J. Korean Math. Soc.
		\rm 55, 1091-1101 (2018).
		
		\bibitem{AYY0}
		{Grigor'yan, A., Lin, Y., Yang, Y.:}
		{Kazdan-Warner equation on graph.}
		\sl Calc. Var. Partial Diff. Equ.
		\rm 55 (4), 92 (2016).
		
		\bibitem{AYY1}
		{Grigor'yan, A., Lin, Y., Yang, Y.:}
		{Yamabe type equations on graphs.}
		\sl J. Diff. Equ.
		\rm 261 (9), 4924-4943 (2016).
		
		\bibitem{AYY2}
		{Grigor'yan, A., Lin, Y., Yang, Y:}
		{Existence of positive solutions to some nonlinear equations on locally finite graphs.}
		\sl Sci. China Math.
		\rm 60 (7), 1311-1324 (2017).
		
		\bibitem{HX}
		{Han, X.:}
		{The existence of multi-vortices for a generalized self-dual Chern-Simons model.}
		\sl Nonlinearity
		\rm 26 (3), 805-835 (2013).
		
		\bibitem{HSM}
		{Han, X.L., Shao, M.Q.:}
		{$p$-Laplacian equations on locally finite graphs.}
		\sl  Acta Math. Sin. (Engl. Ser.)
		\rm 37 (11), 1645-1678 (2021).
		
		\bibitem{HSZ}
		{Han, X.L., Shao, M.Q., Zhao, L.:}
		{Existence and convergence of solutions for nonlinear biharmonic equations on graphs.}
		\sl  J. Diff. Equ.
		\rm 268 (7), 3936-3961 (2020).
		
		\bibitem{HKP}		
		{Hong, J., Kim, Y., Pac, P.:} {Multivortex solutions of the abelian Chern-Simons-Higgs theory.}
		\sl Phys. Rev. Lett.
		\rm 64 (19), 2230-2233 (1990).
		
		\bibitem{HLLY}
		{Horn, P., Lin, Y., Liu, S., Yau, S.T.:}
		{Volume doubling, Poincar$\acute{e}$ inequality and Gaussian heat kernel estimate for non-negatively curved graphs.}
		\sl  J. Reine Angew. Math.
		\rm 757, 89-130 (2019).
		
		\bibitem{HQ}
		{Hou, S., Qiao, W.:}
		{Solutions to a generalized Chern-Simons Higgs model on finite graphs by topological degree.}
		\sl J. Math. Phys.
		\rm 65, 081503 (2024).
		
		\bibitem{HS}
		{Hou, S., Sun, J.:}
		{Existence of solutions to Chern-Simons Higgs equations on graphs.}
		\sl Calc. Var. Part. Diff. Equ.
		\rm 61 (4), 139 (2022).
		
		\bibitem{HXW}
		{Hua, B., Xu, W.:}
		{Existence of ground state solutions to some nonlinear Schr\"odinger equations on lattice graphs.}
		\sl Calc. Var. Partial Differ. Equ.
		\rm 62 (4), 127 (2023).	
		
		\bibitem{HLY}
		{Huang, A., Lin, Y., Yau, S.T.:}
		{Existence of solutions to mean field equations on graphs.}
		\sl Comm. Math. Phys.
		\rm 377, 613-621 (2020).
		
		\bibitem{HX1}
		{Huang, X.:}
		{On uniqueness class for a heat equation on graphs.}
		\sl  J. Math. Anal. Appl.
		\rm  393, 377-388 (2012).	
		
		\bibitem{JW}	
		{Jackiw, R., Weinberg, E.:} {Self-dual Chern-Simons vortices.}
		\sl Phys. Rev. Lett.
		\rm 64 (19), 2234-2237 (1990).
		
		
		\bibitem{LSY}
		{Li, J., Sun, L., Yang, Y:}
		{Topological degree for Chern-Simons Higgs models on finite graphs.}
		\sl Calc. Var. Part. Diff. Equ.
		\rm 63, 81 (2024).
		
		\bibitem{LYY1}
		{Li,Y.Y.:}
		{Harnack type inequality: the method of moving planes.}
		\sl Comm. Math. Phys.
		\rm 200, 421-444 (1999).
		
		\bibitem{LPY}	
		{Lin, C.S., Ponce, A.C., Yang, Y.S.:} { A system of elliptic equations arising in Chern-Simons field theory.}
		\sl J. Funct. Anal.
		\rm 247, 289-350 (2007).
		
		\bibitem{LWY}
		{Lin, Y., Wu, Y.T.:}
		{The existence and nonexistence of global solutions for a semilinear heat equation on graphs.}
		\sl Calc. Var. Partial Differ. Equ.
		\rm 56 (4), 102 (2017).
		
		\bibitem{LWY1}
		{Lin, Y., Wu, Y.T.:}
		{Blow-up problems for nonlinear parabolic equations on locally finite graphs.}
		\sl Acta Math. Scientia
		\rm  38B (3), 843-856 (2018).
		
		
		\bibitem{L}
		{Liu, Y.:}
		{Brouwer degree for mean field equation on graph.}
		\sl Bull. Korean Math. Soc.
		\rm 59 (5), 1305-1315 (2022).
		
		\bibitem{LY}
		{Liu, Y., Yang, Y.:}
		{Topological degree for Kazdan-Warner equation in the negative case on finite graph.}
		\sl Ann. Glob. Anal. Geom.
		\rm 65, 29 (2024).
		
		
		\bibitem{NT1}	
		{Nolasco, M., Tarantello, G.:} { On a sharp type inequality on two dimensional compact manifolds.}
		\sl Arch.
		Rational Mech. Anal.
		\rm 145, 161-195 (1998).
		
		\bibitem{NT2}
		{Nolasco, M., Tarantello, G.:} { Double vortex condensates in the Chern-Simons-Higgs theory.}
		\sl Calc. Var.
		Part. Differ. Equ.
		\rm 9, 31-94 (1999).
		
		\bibitem{SW}
		{Sun, L., Wang, L.:}
		{Brouwer degree for Kazdan-Warner equations on a connected finite graph.}
		\sl Adv. Math.
		\rm 404 (Part B), 108422 (2022).
		
		\bibitem{T}
		{Tarantello, G.:}
		{Multiple condensate solutions for the Chern-Simons-Higgs theory.}
		\sl J. Math. Phys.
		\rm 37 (8), 3769-3796 (1996).	
		
		
		
		\bibitem{W}
		{Wang, R.:}{ The existence of Chern-Simons vortices.} \sl Comm. Math. Phys. \rm 137, 587-597 (1991).
		
		
		\bibitem{ZZ}
		{Zhang, N., Zhao, L.:}
		{Convergence of ground state solutions for nonlinear Schr\"odinger equations on graphs.}
		\sl Sci. China Math.
		\rm 61 (8), 1481-1494 (2018).
		
	\end{thebibliography}
\end{document}